\documentclass[12pt]{article}
\usepackage{amsmath}
\usepackage{amsthm}
\usepackage{amssymb}
\usepackage{graphicx}
\theoremstyle{plain}

\newtheorem{theorem}{Theorem}

\begin{document}

\title{A weighted interpretation for the super Catalan numbers}    

\author{Emily\ Allen \thanks{Department of Mathematical Sciences, Carnegie Mellon University, Pittsburgh, PA 15213, {\tt <eaallen@andrew.cmu.edu>}}\\
 Irina\ Gheorghiciuc \thanks{Department of Mathematical Sciences, Carnegie Mellon University, Pittsburgh, PA 15213, {\tt <gheorghi@andrew.cmu.edu>}}}
\date{}

\maketitle

\begin{abstract}
The super Catalan numbers $T(m,n)=(2m)!(2n)!/2m!n!(m+n)!$ are integers which generalize the Catalan numbers. With the exception of a few values of $m$, no combinatorial interpretation in known for $T(m,n)$. We give a weighted interpretation for $T(m,n)$ and develop a technique that converts this weighted interpretation into a conventional combinatorial interpretation in the case $m=2$.
\end{abstract}

\section{Introduction}

As early as 1874 Eugene Catalan observed that the numbers
\[ S(m,n) = \frac{{2m \choose m}{2n \choose n}}{{m + n \choose n}} = \frac{(2m)!(2n)!}{m!n!(m+n)!}\]
are integers. This can be proved algebraically by showing that, for every prime number $p$, the power of $p$ which divides $m!n!(m+n)!$ is at most the power of $p$ which divides $(2m)!(2n)!$. No combinatorial interpretation of $S(m,n)$ is yet known.

Interest in the subject in the modern era was reignited by Gessel \cite{SuperBallot}. He noted that, except for $S(0,0)$, the numbers $S(m,n)$ are even. Gessel refers to
\[T(m,n) = \frac{(2m)!(2n)!}{2[m!n!(m+n)!]}\]
as the super Catalan numbers.

Clearly $T(0,n)=\frac{1}{2} {2n \choose n}$, whilst $T(1,n) = C_n$ giving the Catalan numbers, a well-known sequence with over 66 combinatorial interpretations \cite{EC}.

An interpretation of $T(2,n)$ in terms of blossom trees has been found by Schaeffer \cite{Schaeffer}, and another in terms of cubic trees by Pippenger and Schleich \cite{P&S}. An interpretation of $T(2,n)$ in terms of pairs of Dyck paths with restricted heights has been found by Gessel and Xin \cite{GesselXin}. They have also provided a description of $T(3,n)$. An interpretation of $T(m,m+s)$ for $0 \leq s \leq 3$ in terms of restricted lattice paths has been given by Chen and Wang \cite{CW}.

A weighted interpretation of $S(m,n)$ based on a specialization of Krawtchouk polynomials has been given by Georgiadis, Munemasa and Tanaka \cite{GMT}. Their interpretation is in terms of lattice paths of length $2m+2n$ with a condition on the $y$-coordinate of the end-point of the $2m^{th}$ step.

In Section 2 we provide a weighted interpretation of $T(m,n)$ for $m,n \geq 1$ in terms of $2$-Motzkin paths of length $m+n-2$, or Dyck paths of length $2m+2n-2$. Since the lattice paths in \cite{GMT} are not Dyck paths, our interpretation is different from the one by Georgiadis, Munemasa and Tanaka. In Section 3 we are able to use our weighted interpretation to re-derive a result by Gessel and Xin \cite{GesselXin}, which we were then able to generalize for super Catalan Polynomials in \cite{PolyArt}.

\section{2-Motzkin Paths}

A 2-Motzkin path of length $n$ is a sequence of $n$ steps, starting at the origin and ending at the point $(n,0)$, where the allowable steps are diagonally \textit{up}, diagonally \textit{down}, and \textit{level}. The \textit{level} steps are either \textit{straight} or \textit{wavy}. Define $\mathcal{M}_n$ to be the set of all 2-Motzkin paths of length $n$. A Dyck path of length $2n$ is a 2-Motzkin path of length $2n$ with no \textit{level} steps.

Given a 2-Motzkin path, the level of a point is defined to be its $y$-coordinate. The height of a path is the maximum $y$-coordinate which the path attains. The height of a path $\pi$ will be denoted $h(\pi)$.

For a fixed $m\geq 0$, we call a 2-Motzkin path $\pi$ positive if the $m^{th}$ step begins on an even level, otherwise $\pi$ is negative. Let $P(m,n)$ be the number of positive 2-Motzkin paths of length $m+n-2$, and $N(m,n)$ be the number of negative 2-Motzkin paths of length $m+n-2$.

There is a well-known bijection between 2-Motzkin paths of length $n-1$ and Dyck paths of length $2n$ \cite{tag}. Given a 2-Motzkin path, read the steps from left to right and do the following replacements:
replace an \textit{up} step with two \textit{up} steps, a \textit{down} step with two \textit{down} steps, a \textit{straight} step with an \textit{up} step followed by a \textit{down} step, and
a \textit{wavy} step with a \textit{down} step followed by an \textit{up} step. The resulting path may touch level $-1$, thus, in addition, add an \textit{up} step to the beginning of the 
resulting path and a \textit{down} step to the end to obtain a Dyck path.

\begin{theorem}\label{mainTheorem}
 For $m,n \geq 1$, the super Catalan number $T(m,n)$ counts the number of positive 2-Motzkin paths minus the number of negative 2-Motzkin paths. That is,
\[T(m,n) = P(m,n) - N(m,n).\]
\end{theorem}

\begin{proof}

The super Catalan numbers satisfy the following identity, attributed to Dan Rubenstein \cite{SuperBallot},
\begin{equation}\label{rec}
 4T(m,n) = T(m+1,n) + T(m,n+1).
\end{equation}

Given a 2-Motzkin path $\pi$ of length $m+n-2$, define the weight of $\pi$ to be $1$ if $\pi$ is positive and $-1$ if $\pi$ is negative.
	
Let $F(m,n)$ be the sum of the weights of all 2-Motzkin paths of length $m+n-2$, that is, $F(m,n) = P(m,n) - N(m,n)$. To prove $F(m,n) = T(m,n)$, we will check the initial condition 
\[F(1,n) = C_n\]
and the recurrence given by Eq. \ref{rec},
\[4F(m,n) = F(m+1,n) + F(m,n+1).\]

For $m=1$, the weight of any 2-Motzkin path of length $n$ is $1$ because the first step always starts at the level $y=0$. Hence $F(1,n)=C_n$, giving the number of 2-Motzkin paths of length $n-1$. 

Next we consider the sum of the weights counted by $F(m,n+1) + F(m+1,n)$. If a 2-Motzkin path of length $m+n-1$ has an \textit{up} or \textit{down} step at step $m$, it will be counted once as a positive path and once as a negative path, and will not contribute to this sum.

Paths of length $m+n-1$ with a \textit{level} step at step $m$ will be counted twice. Let $\pi$ be such a 2-Motzkin path. By contracting the $m^{th}$ step in $\pi$, we obtain a 2-Motzkin path of length $m+n-2$; furthermore, every 2-Motzkin path of length $m+n-2$ can be obtained by contracting exactly two 2-Motzkin paths of length $m+n-1$, one with a \textit{wavy} step at step $m$ and one with a \textit{straight} step at step $m$. 

Thus the sum of the weights counted by $F(m,n+1) + F(m+1,n)$ is twice the sum of the weights of 2-Motzkin paths of length $m+n-1$ with \textit{level} steps at step $m$; which is four times the sum of the weights of 2-Motzkin paths of length $m+n-2$, that is, $4F(m,n)$.
\end{proof}

\begin{figure}[h]
\centering
 \includegraphics[scale=.5]{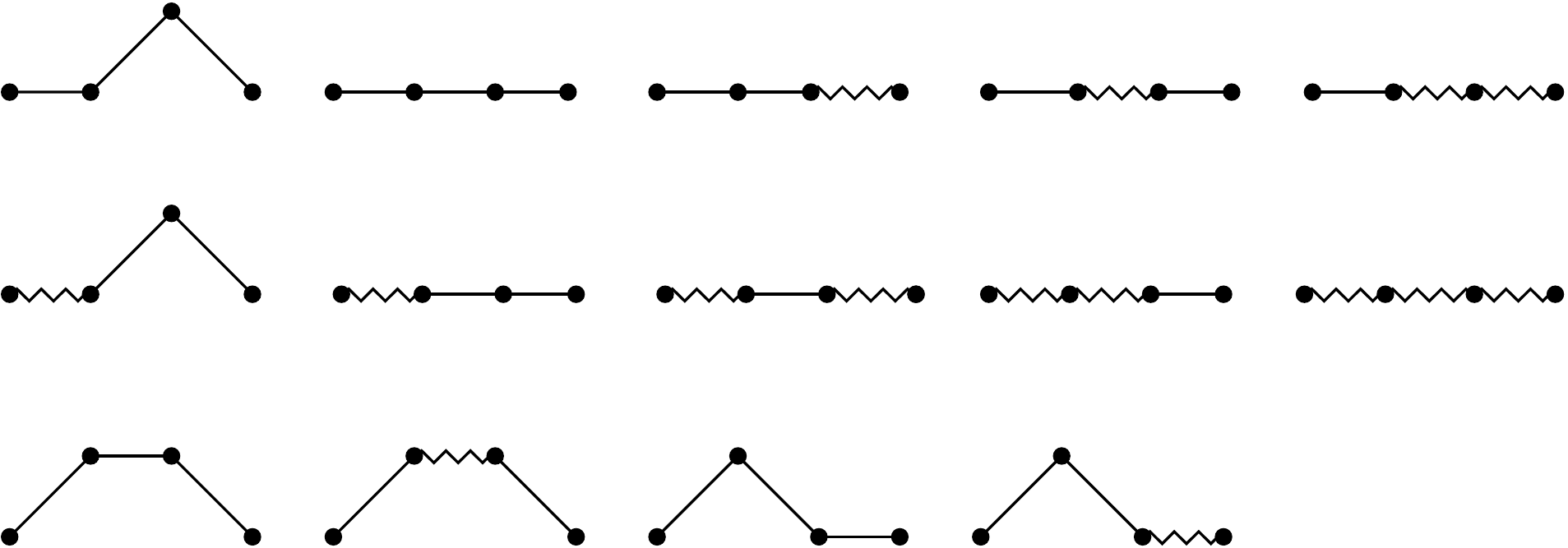}
\caption{When m=2, there are ten positive 2-Motzkin paths and four negative 2-Motzkin paths of length 3. $T(2,3) = P(2,3) - N(2,3) = 6.$}
\end{figure}

This weighted interpretation can be used to prove combinatorially that $T(m,n)=T(n,m)$. Let $\pi$ be a path of length $m+n-2$ counted by $T(m,n)$. Consider the reverse of a path to be that path read from right to left. Since the $m^{th}$ step of $\pi$ and the $n^{th}$ step of the reverse of $\pi$ start at the same point, mapping a path to its reverse is a weight preserving involution between the 2-Motzkin paths counted by $T(m,n)$ and the 2-Motzkin paths counted by $T(n,m)$.

We can reformulate the result of Theorem \ref{mainTheorem} in terms of Dyck paths. In this case $P(m,n)$ is the number of Dyck paths of length $2m+2n - 2$ whose $2m-1^{st}$ step ends on level $1 \pmod 4$, and $N(m,n)$ is the number of Dyck paths of length $2m+2n-2$ whose $2m-1^{st}$ step ends on level $3 \pmod 4$. Let $B(n,r)$ be the number of ballot paths that start at the origin, end at the point $(2n-1,2r-1)$, and do not go below the $x$-axis. It is well known that $B(n,r) = \frac{r}{n}{2n \choose n+r}$. Then

\begin{equation}
T(m,n) = \sum_{r \geq 1} (-1)^{r-1}B(m,r)B(n,r)
\end{equation}
and
\begin{equation}\label{balloteq}
T(m,n) = \sum_{r\geq 1} (-1)^{r-1}\frac{r^2}{nm}{2m \choose m+r}{2n \choose n+r}.
\end{equation}

Eq. \ref{balloteq} is a new identity for the super Catalan number $T(m,n)$. We provide a $q$-analog of it in \cite{PolyArt}, its algebraic proof appears in \cite{thesis}.

\section{Combinatorial Techniques}

In \cite{GesselXin} Gessel and Xin use an inclusion-exclusion argument to prove the following result.

\begin{theorem}[Gessel, Xin] For $n \geq 1$, the number $T(2,n)$ counts the ordered pairs of Dyck paths $(\pi,\rho)$ of total length $2n$ with $|h(\pi) - h(\rho)| \leq 1$. Here $\pi$ and $\rho$ are allowed to be the empty path. The height of the empty path is zero. 
\label{Gessel}
\end{theorem}

Our goal in this section is to derive a similar result using Theorem \ref{mainTheorem} and some direct Dyck paths subtraction techniques that will be easier to generalize for larger values of $m$. We already were able to generalize this result to super Catalan Polynomials in \cite{PolyArt}.

Let $\mathcal{D}_{n}$ denote the set of Dyck paths of length $2n$. For a path $\pi \in \mathcal{D}_n$, let $X$ be the last, from left to right, level one point up to and including the right-most maximum $R$ on $\pi$. Let $h_{-}(\pi)$ denote the maximum level that the path $\pi$ reaches from its beginning until and including point $X$, and $h_{+}(\pi)$ denote the maximum level that the path $\pi$ reaches after and including point $X$. Obviously $h_-(\pi) \leq h_+(\pi) = h(\pi)$.

\begin{theorem} Let $n \geq 1$. The super Catalan number $T(2,n)$ counts Dyck paths $\pi$ of length $2n$ such that $h_{+}(\pi) \leq h_{-}(\pi)+2$, the path of height one counting twice. 
\label{T(2,n)}
\end{theorem}

\begin{proof} Let $\mathcal{A}_{n}$ denote the set of Dyck paths of length $2n$ that start with {\it up, down, up}, $\mathcal{B}_{n}$ denote the set of Dyck paths of length $2n$ that start with {\it up, up, down}, and $\mathcal{N}_{n}$ denote the set of Dyck paths of length $2n$ that start with {\it up, up, up}.

By Theorem \ref{mainTheorem}, $T(2,n)=P(2,n)-N(2,n)$, where $P(2,n)$ is the number of 2-Motzkin paths of length $n$ that start with a level step, and $N(2,n)$ is the number 2-Motzkin paths of length $n$ that start with an up step. The canonical bijection between 2-Motzkin paths and Dyck paths leads to the following interpretation: $$T(2,n)=|\mathcal{A}_{n+1}|+|\mathcal{B}_{n+1}|-|\mathcal{N}_{n+1}|.$$

By contracting the second and third steps in the paths in $\mathcal{A}_{n+1}$ and $\mathcal{B}_{n+1}$ we get twice $\mathcal{D}_{n}$, so $|\mathcal{A}_{n+1}|=|\mathcal{B}_{n+1}|=C_n$. 

We consider all paths $\pi$ in $\mathcal{N}_{n+1}$ that do not attain level one between the third step of $\pi$ and the right-most maximum point $R$ on $\pi$. The set of all such paths will be denoted by $\mathcal{N}^*_{n+1}$. Let $\mathcal{N}^{**}_{n+1}=\mathcal{N}_{n+1}-\mathcal{N}^*_{n+1}$. Then $$T(2,n)=2|\mathcal{D}_{n}|-|\mathcal{N}^*_{n+1}|-|\mathcal{N}^{**}_{n+1}|.$$

First we establish an injection $f$ from $\mathcal{N}^*_{n+1}$ to $\mathcal{D}_n$. For $\pi \in \mathcal{N}^*_{n+1}$, let $RQ$ be the {\it down} step that follows the right-most maximum point $R$ of $\pi$. We define $f(\pi)$ to be the path obtained by removing the second and third steps in $\pi$, both of which are {\it up} steps, and then substituting the {\it down} step $RQ$ by an {\it up} step. See Figure~\ref{Mn+1}. Since $\pi$ does not attain level one between its third step and $R$, $f(\pi)$ is a Dyck path of length $2n$. Note that $Q$ is the left-most maximum on $f(\pi)$. Also, since at least two {\it up} steps precede $Q$ on $f(\pi)$, the height of $f(\pi)$ is at least two. Thus the Dyck path of height one and length $2n$ is not in the image of $f$.

\begin{figure}[!ht]
\begin{center}
\includegraphics[scale=0.5]{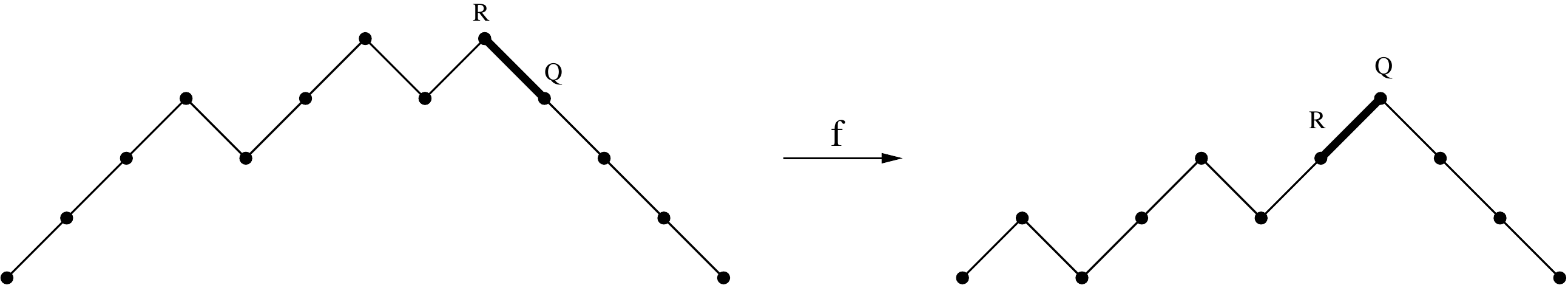}
\caption{$f$ removes the $2^{nd}$ and $3^{rd}$ steps, substitutes the {\it down} step $RQ$ by an {\it up} step}
\label{Mn+1}
\end{center}
\end{figure}

We will show that $f$ is an injection and that the only path in $\mathcal{D}_n$ that is not in the image of $f$ is the Dyck path of height one. Let $\rho$ be in $\mathcal{D}_n$ of height $h(\rho)>1$. Let $Q$ be the left-most maximum on $\rho$ and $RQ$ be the {\it up} step that precedes $Q$. Insert two {\it up} steps after the first step of $\rho$, then substitute the {\it up} step $RQ$ by a {\it down} step, which makes $R$ the right-most maximum of the resulting path $\pi$. The path $\pi$ is in $\mathcal{N}^*_{n+1}$ and $f(\pi)=\rho$.

It follows that $|\mathcal{D}_{n}|-|\mathcal{N}^*_{n+1}|$ counts only one path, the Dyck path of length $2n$ and height one.  

Next we establish an injection $g$ from $\mathcal{N}^{**}_{n+1}$ to $\mathcal{D}_n$. A path $\pi$ in $\mathcal{N}^{**}_{n+1}$ attains level one between its third step and the right-most maximum point $R$ on $\pi$. Let $Y$ be the first point between the third step of $\pi$ and $R$ at which $\pi$ attains level one. The segment $XY$ that consists of two {\it down} steps precedes $Y$. We remove the second and third steps of $\pi$ and substitute the two {\it down} steps $XY$ by two {\it up} steps. See Figure~\ref{M1}. The resulting path is a ballot path of length $2n$ that ends at level two. From left to right, $X$ is the last level one point on this ballot path. The maximum level that this path reaches up to and including point $X$ is less than the maximum level it reaches after and including point $X$ by at least 4. 

\begin{figure}[!ht]
\begin{center}
\includegraphics[scale=0.5]{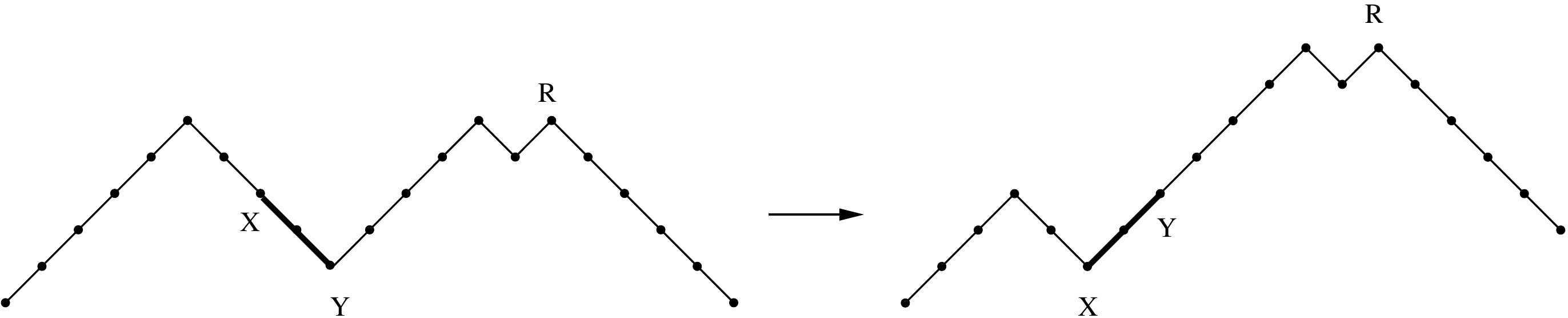}
\caption{First part of $g$ action is removing the $2^{nd}$ and $3^{rd}$ steps, substituting the two {\it down} steps $XY$ by two {\it up} steps}
\label{M1}
\end{center}
\end{figure}

Let $L$ be the left-most maximum point of this ballot path and $ML$ be the {\it up} step that precedes $L$. Substitute the {\it up} step $ML$ by a {\it down} step. See Figure ~\ref{M2}. The resulting path $g(\pi)$ is in $\mathcal{D}_n$ and $M$ is its right-most maximum. Note that $X$ is the last level one point on $g(\pi)$ before its right-most maximum $M$ and $h_{+}(g(\pi)) \geq h_{-}(g(\pi))+3$. 

\begin{figure}[!ht]
\begin{center}
\includegraphics[scale=0.5]{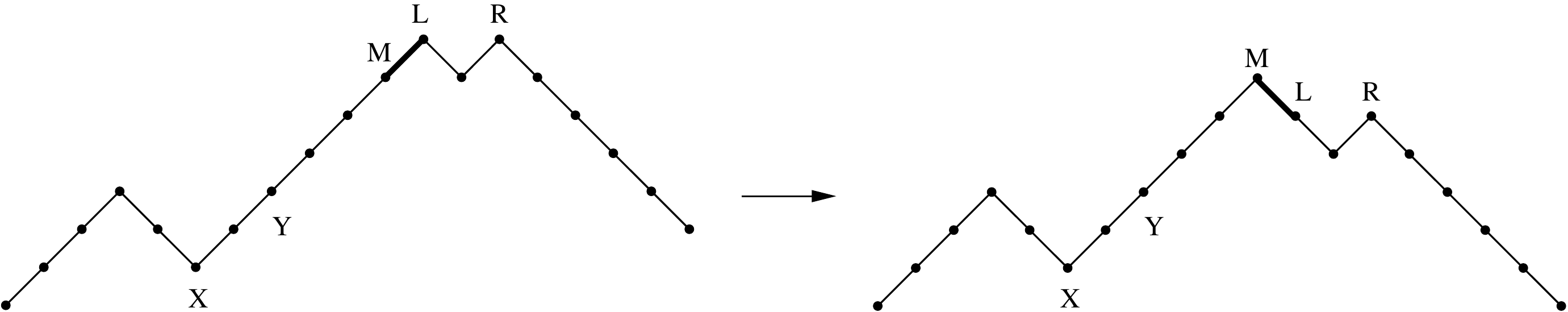}
\caption{Second part of $g$ action is substituting the {\it up} step $ML$ with a {\it down} steps}
\label{M2}
\end{center}
\end{figure}

We will show that $g$ is an injection and that the only paths in $\mathcal{D}_n$ that are not in the image of $g$ are the Dyck paths $\sigma$ that satisfy $h_{+}(\sigma) \leq h_{-}(\sigma)+2$. Let $\rho$ be in $\mathcal{D}_n$ and $h_{+}(\rho) \geq h_{-}(\rho)+3$. Let $M$ be the right-most maximum on $\rho$ and $ML$ be the {\it down} step that follows $M$. Substitute the {\it down} step $ML$ by an {\it up} step. The result is a ballot path of length $2n$ that ends at level two. Note that $L$ is the left-most maximum on this ballot path. Let $R$ denote the right-most maximum on this ballot path. From left to right, $X$ is the last level one point on this ballot path. The maximum level that this path reaches up to and including point $X$ is less than the maximum level it reaches after and including point $X$ by at least 4.  Since $X$ is the last level one point, it is followed by the segment $XY$ that consists of two {\it up} steps. Next we insert two {\it up} steps after the first step of this ballot path and then substitute the two {\it up} steps $XY$ by two {\it down} steps. The resulting path is a Dyck path of length $2n+2$, we denote it by $\pi$. Point $Y$ is the first level one point after the third step of $\pi$. Note that the maximum level that this Dyck path reaches after $Y$ is at least the maximum level that this Dyck path reaches up to and including $Y$, which means that the right-most maximum $R$ is to the right of $Y$. If follows that $p \in \mathcal{N}^{**}_{n+1}$ and $g(\pi)=\rho$.

Thus $|\mathcal{D}_{n}|-|\mathcal{N}^{**}_{n+1}|$ counts Dyck paths $\pi$ that satisfy $h_{+}(\pi) \leq h_{-}(\pi)+2$. 

\end{proof}

We will now show a simple bijection between the objects described in Theorem~\ref{T(2,n)} and Theorem ~\ref{Gessel}.

\begin{figure}[!ht]
\begin{center}
\includegraphics[scale=0.5]{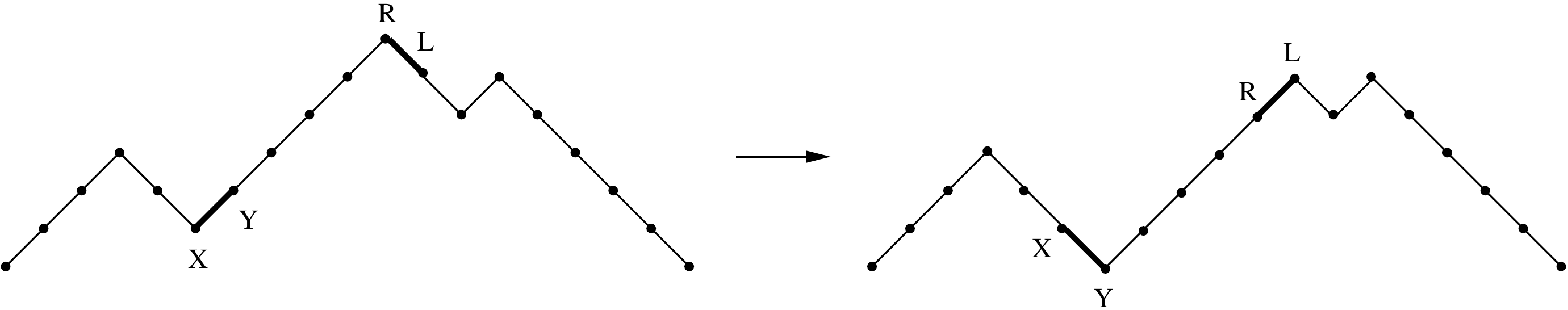}
\caption{From Dyck paths described in Theorem~\ref{T(2,n)} to pairs of Dyck path described in Theorem~\ref{Gessel}}
\label{gessel}
\end{center}
\end{figure}

Let $\pi$ be a Dyck path of length $2n$ and height $h(\pi)>1$, such that $h_{+}(\pi) \leq h_{-}(\pi)+2$. Let $R$ be the right-most maximum of $\pi$. Note that $X$ is followed by an {\it up} step $XY$ and $R$ is followed by a {\it down} step $RL$. Substitute the {\it up} step $XY$ with a down step, substitute the {\it down} step $RL$ with an {\it up} step. See Figure~\ref{gessel}.
As a result, the portion of $\pi$ between $Y$ and $R$ will be lowered by two levels. Since $\pi$ does not attain level one between $Y$ and $R$, the resulting path is a Dyck path with point $Y$ on level zero. 

Note that $Y$ separates this Dyck path into a pair of Dyck paths $(\rho,\sigma)$. The height of $\rho$ is $h_{-}(\pi)$, the height of $\sigma$ is $h_{+}(\pi)-1$. Thus $|h(\rho) - h(\sigma)| \leq 1$. Since $L$ is the left-most maximum on $\sigma$, this mapping is reversible. Theorem~\ref{T(2,n)} counts the Dyck path $\tau$ of height one twice. This corresponds to the pairs $(\tau, \epsilon)$ and $(\epsilon, \tau)$ in Theorem~\ref{Gessel}, where $\epsilon$ is the empty path.

\section*{Acknowledgements}

We are very thankful to Ira Gessel for his helpful comments. The problem about the combinatorial interpretation of the super Catalan Numbers was first mentioned to us by the late Herb Wilf, to whom we are immensely grateful.


\begin{thebibliography}{1}

\bibitem{thesis}
E. Allen, Combinatorial Interpretations of Generalizations of Catalan Numbers and Ballot Numbers, 2014, Ph.D. Thesis.

\bibitem{PolyArt}
E. Allen and I. Gheorghiciuc, On super Catalan polynomials, arXiv:1403.5296.

\bibitem{CW}
X. Chen and J. Wang, The super Catalan numbers $S(m,m+s)$ for $s \leq 3$ and some integer factorial ratios, http://www.math.umn.edu/~reiner/REU/ChenWang2012.pdf, 2012.

\bibitem{GMT}
E. Georgiadis, A. Munemasa and H. Tanaka, A note on super Catalan numbers, Interdiscip. Inform. Sci. 18 (2012),  23-24.

\bibitem{SuperBallot}
I. Gessel, Super ballot numbers, J. Symb. Comput. 14 (1992), 179-194.  
   
\bibitem{GesselXin}
I. Gessel and  G. Xin, A combinatorial interpretation of the numbers $6(2n)!/n!(n+2)!$, J. Integer Seq. 8 (2005). 

\bibitem{tag}
M. Pierre-Delest and G. Viennot, Algebraic languages and polyominoes enumeration, Theoret. Comput. Sci. 34 (1984), 196-206.

\bibitem{P&S}
N. Pippenger and K. Schleich, Topological characteristics of random triangulated surfaces, Random Structures and Algorithms 28 (2006), 247-288.

\bibitem{Schaeffer}
G. Schaeffer, A combinatorial interpretation of super-Catalan numbers of order two, unpublished manuscript, 2003.
  
\bibitem{EC}   
R. Stanley, Enumerative Combinatorics, vol. 2, Cambridge University Press, 1998.     
   
\end{thebibliography}
\end{document}